\documentclass[12pt]{amsart}

\usepackage{amsfonts,amsmath,latexsym,amssymb,verbatim,amsbsy,amsthm}

\usepackage[top=1in, bottom=1in, left=1in, right=1in]{geometry}

\usepackage[dvipsnames]{xcolor}

\usepackage[colorlinks=true, pdfstartview=FitV, linkcolor=RoyalBlue,citecolor=ForestGreen, urlcolor=blue]{hyperref}

\theoremstyle{plain}
\newtheorem{THEOREM}{Theorem}[section]

\newtheorem{theorem}[THEOREM]{Theorem}

\newtheorem{lemma}[THEOREM]{Lemma}

\theoremstyle{definition}

\theoremstyle{remark}

\newcommand{\thm}[1]{Theorem~\ref{#1}}

\def \a {\alpha}

\def \g {\gamma}
\def \d {\delta}
\def \e {\varepsilon}

\def \n {\nabla}

\def \th {\theta}

\def \cL {\mathcal{L}}
\def \cM {\mathcal{M}}

\def \cP {\mathcal{P}}

\newcommand{\N}{\ensuremath{\mathbb{N}}}   %%% naturals
\newcommand{\Z}{\ensuremath{\mathbb{Z}}}   %%% integers
\newcommand{\R}{\ensuremath{\mathbb{R}}}   %%% reals
\newcommand{\T}{\ensuremath{\mathbb{T}}}   %%% torus

\def \p {\partial}
\def \ra {\rightarrow}

\renewcommand{\geq}{\geqslant}

\renewcommand{\leq}{\leqslant}

\def \fmin {\phi_{min}}

\DeclareMathOperator{\tr}{Tr} %
\def \lowr {\underline{\rho}}
\def \uppr {\overline{\rho}}

\def \dx  {\, \mbox{d}x}

\def \dz  {\, \mbox{d}z}
\def \ds  {\, \mbox{d}s}

\def \dth  {\, \mbox{d}\th}

\def \ddt  {\frac{\mbox{d\,\,}}{\mbox{d}t}}

\begin{document}

\title[Global existence of nearly aligned flocks]{Global existence and stability of nearly aligned flocks}

\author{Roman Shvydkoy}
\address{Department of Mathematics, Statistics, and Computer Science, M/C 249\\
    University of Illinois, Chicago, IL 60607, USA}
\email{shvydkoy@uic.edu}

\date{\today}



\thanks{\textbf{Acknowledgment.} Research was supported in part by NSF grant DMS 1515705, and the College of LAS at UIC. The author thanks Eitan Tadmor for stimulating discussions}

\begin{abstract}
We study regularity of a hydrodynamic singular model of collective behavior introduced in \cite{ST1}.  In this note we address the question of global well-posedness in multi-dimensional settings. It is shown that any initial data $(u,\rho)$ with small velocity variations $|u(x) - u(y)| < \e$ relative to its higher order norms, gives rise to a unique global regular solution which aligns and flocks exponentially fast.  Moreover, we prove that the  limiting flocks are stable.
\end{abstract}

\maketitle

\section{Introduction}
In this note we study a hydrodynamic model of collective behavior given by
\begin{equation}\label{e:main}
\left\{
\begin{split}
\rho_t + \n \cdot (\rho u) & = 0, \\
u_t + u \cdot \n u &= [\cL_\phi, u](\rho),
\end{split} \right. \qquad (x,t)\in \T^n \times \R_+.
\end{equation}
where the  commutator $[\cL_\phi, u](\rho):=\cL_\phi(\rho u)-\cL_\phi(\rho) u$
involves a non-negative communication kernel $\phi$, and $\cL_\phi$ is given by 
\begin{equation}\label{e:L}
\cL_\phi(f):=\int_{\T^n} \phi(|x-y|)(f(y)-f(x))dy.
\end{equation}
This model represents a macroscopic description of an agent-based Cucker-Smale dynamics. The basic objective of such models is to explain emergence of flocking and alignment in a wide range of biological and social systems, see \cite{CCTT2016,CS2007a,MT2014,PS2012} and references therein.  The system is driven by a  self-organization process where agents exert influence on each other's momenta through the kernel $\phi$. The long time behavior is characterized by  alignment, i.e. convergence of velocities $u$ to a single value $\bar{u}$, and flocking, which in our context is defined by convergence of the density $\rho$ to a traveling wave $\rho_\infty(x - t \bar{u})$. A common deficiency of existing models is the use of long range connections expressed, for example, by  $\int_1^\infty \phi(r) dr =\infty$. Such connections, albeit necessary for the alignment, should be less dominant in a realistic system. In a new class of models  introduced by E. Tadmor and the author in \cite{ST3,ST1,ST2} and independently by T. Do, et al \cite{DKRT2017}, we proposed to use  the singular kernel of the fractional Laplacian given by 
\begin{equation}\label{e:kernel}
\phi_\a(x) := \sum_{k\in \Z^n} \frac{1}{|x+ 2\pi k|^{n+\a}}, \qquad 0<\a< 2.
\end{equation}
which puts more emphasis on local interactions.  Global well-posedness theory for these singular models has been developed only in 1D mainly due to presence of an additional conserved quantity 
\[
e = u_x + \cL_\phi \rho.
\]
It establishes control over higher order terms by means of a pointwise bound $|e| < C \rho$.  In multi-dimensional settings the corresponding quantity is given by 
\[
e = \n \cdot u + \cL_\phi \rho,
\]
and satisfies 
\begin{equation}\label{e:e}
e_t + \nabla \cdot(u e)  = (\n \cdot u)^2 -  \tr(\n u)^2 ,
\end{equation}
see \cite{HeT2017} and Section~\ref{s:proof} below for derivation. Lack of control on $e$ in this case is part of the reason why in multiple dimensions the model has no developed regularity theory. The two notable exceptions are the result of Ha, et al \cite{Ha2014} demonstrating global existence in the case of smooth communication kernel $\phi$ with small initial data in higher order Sobolev spaces, $\|u\|_{H^{s+1}} < \e_0$, where $\e_0$ depends on $\|\rho_0\|_{H^s}$; and  He and Tadmor, \cite{HeT2017}  with global existence in 2D under a smallness assumption only on the initial amplitude of $u$, spectral gap of the stretch tensor $\n u + \n^\top u$, and under the threshold condition $e\geq 0$ (also necessary in 1D for smooth kernels, \cite{CCTT2016}).

The goal of this note is to address  global existence of smooth solutions for  singular models \eqref{e:main} -- \eqref{e:kernel} in any dimension under periodic boundary conditions. We establish a small initial data result where smallness is expressed only in terms of the initial amplitude of the solution. To be precise, let us introduce some notation and terminology.  We define the amplitude by
\begin{equation}
A(t) = \sup_{x \in \T^n} |u(x,t) - \bar{u}|,
\end{equation}
where $\bar{u} = \cP/ \cM$, $\cP = \int_{\T^n} u\rho \dx$, $\cM = \int_{\T^n} \rho \dx$. Observe that the momentum $\cP$ and mass $\cM$ are conserved. 
Exactly the same argument as in \cite{ST1} applied to each component of $u$ proves the following a priori bound:
\begin{equation}\label{e:A}
A(t) \leq A_0 e^{- \fmin \cM t}, \quad  \fmin = \min_{x\in \T^n} \phi(x).
\end{equation}
Hence, global solutions automatically align exponentially fast. Note that $(\bar{u}, \bar{\rho})$, where $\bar{\rho} = \rho_\infty(x - t \bar{u})$ and $\bar{u}$ is constant, is a traveling wave solution to \eqref{e:main}. We call it a flocking state. As shown in \cite{ST3,ST2} any solution to \eqref{e:main} in 1D for any $0<\a<2$ converges exponentially fast to some flocking state in smooth regularity classes. 

We use $|\cdot|_X$ to denote standard metrics and $[\cdot]_s$ to denote the homogeneous metric of $\dot{W}^{s,\infty}$. We now state our main result.

\begin{theorem}\label{t:main}
Let $0 < \a < 2$.  There exists an $N\in \N$ such that for any sufficiently large $R>0$ any initial condition $(u_0,\rho_0) \in H^{m}(\T^n) \times H^{m-1+\a}(\T^n)$, $m\geq n+4$, satisfying
\begin{equation}
\begin{split}
     |\rho_0|_\infty, |\rho^{-1}_0|_\infty, [u_0]_3, [\rho_0]_3  & \leq R,\\
     A_0 & \leq \frac{1}{R^N},
\end{split}
\end{equation}
gives rise to a unique global solution in  class  $C([0,\infty): H^{m} \times H^{m-1+\a})$. Moreover, the solution converges to a flocking state exponentially fast at least in $C^1$:
\[
|\rho(t) - \bar{\rho}(t)|_{C^1} < C e^{-\d t}.
\]
\end{theorem}

Finally we note that the argument of \thm{t:main} establishes  a uniform control on $C^2$-norms of $u$ and the distance between the initial density $\rho_0$ and its final distribution  $\rho_\infty$.  As a consequence we obtain a stability result for flocking states.  

\begin{theorem}\label{t:stab}
    Let $(\bar{u},\bar{\rho})$ be a flocking state, where $\bar{\rho}(x) = \rho_\infty(x - t \bar{u})$, and let $(u_0,r_0)$ be an initial data satisfying 
the conditions of \thm{t:main}.  Suppose $|u_0 - \bar{u}|_\infty + |r_0 - \rho_\infty|_\infty < \e$.  Then the solution will converge to another flock $r_\infty$ with $|r_\infty - \rho_\infty|_\infty < \e^\th$, where $\th\in (0,1)$ depends only on $\a$. 
\end{theorem}

\section{Proof of the main result}\label{s:proof}

\noindent \textsc{Local existence and regularity criterion}. First, to see \eqref{e:e}, we take the divergence of the momentum equation, denoting $d = \n \cdot u$, 
\[
d_t+u\cdot \nabla d + \tr(\n u)^2 = \cL_\phi( \n \cdot (u \rho))  - u \cdot \n \cL_\phi \rho - d  \cL_\phi \rho.
\]
Replacing $\cL_\phi( \n \cdot (u \rho)) = - \cL_\phi  \rho_t$ and collecting the terms we obtain \eqref{e:e}.

 The proof of local existence in space $u,\rho \in H^m$ for $m\geq n+4$ carries over ad verbatim from 1D case, see  \cite{ST1}. Although the right hand side of \eqref{e:e} is not zero, it is quadratic in $\n u$.  This results in the same a priori bound on the $e$-quantity:
\begin{equation}\label{eq:leb} 
\p_t |e|_{H^{m-1}}^2 \leq C (|e|_{H^{m-1}}^2 + |u|_{H^{m}}^2)( |\n u|_\infty + |e|_\infty).
\end{equation}
The bound on $u$ is also the same as in \cite{ST1}. We still have the Riccati boun for $Y = |u|_{H^m} + |e|_{H^{m-1}} + |\rho|_2 \sim  |u|_{H^m} + |\rho|_{H^{m-1+\a}}$ :
\[
Y_t \leq C Y^2.
\]
Along with the local existence estimates comes a Beale-Kato-Mayda type regularity criterion:  as long as $|\n u|_\infty$ is bounded on an interval $[0,T]$ all the higher order norms remain bounded as well, and hence the solution can be extended beyond $T$. The only difference being the $\n u$-term in the e-equation, which under BKM condition remains bounded, and hence so is $|e|_\infty$, which closes the estimate \eqref{eq:leb}.

In order to control $|\n u|_\infty$ it is sufficient to establish control over a higher order H\"older metric. We choose to work with the $C^{2+\g}$-norm as opposed to, say, $C^{1+\g}$-norm for two technical reasons. First, it will be necessary to overcome the singularity of the kernel for values of $\a$ greater than $1$, for which  $C^{1+\g}$ is insufficient. Second, as a byproduct we establish control over  $|\n^2 u|_\infty$ as well, which will lead us to the proof of strong flocking as in \cite{ST2}. 
 
From now on we will fix an exponent $0<\g<1$ to be identified later but dependent only on $\a$.  The following notation will be used throughout:
\[
\begin{split}
\d_h u(x) &= u(x+h) - u(x), \quad \tau_z u(x)  = u(x+z) \\
\d^2_h u & = \d_h(\d_h u), \quad \d^3_h u = \d_h(\d_h(\d_h u)).
\end{split}
\]
We use the H\"older metric defined via third order finite differences 
\begin{equation}\label{e:holder}
[u]_{2+\g} = \sup_{x,h \in \T^n}  \frac{|\d^3_h u(x)|}{|h|^{2+\g}}.
\end{equation}
The equivalence of \eqref{e:holder} to the classical norm $[\n^2 u]_\g$ is a well known result in approximation theory, see \cite{Triebel}.

\medskip

\noindent \textsc{Breakthrough scenario}. We assume that we are given a local solution $(u,\rho) \in C([0,\infty): H^{m} \times H^{m-1+\a})$ satisfying the assumptions of the Theorem.  Note that in view of the smallness assumption on $A_0$, the norm $[u(t) ]_{2+\g}$ will remain smaller than $1$ at least for a short period of time. We thus study a possible critical time $t^*<T$ at which the solution  reaches size $R$ for the first time:
\begin{equation}\label{e:R}
    [u(t^*) ]_{2+\g} = R , \quad [u(t) ]_{2+\g} < R,  \quad t< t^*.
\end{equation}
A contradiction will be achieved if we show that $\p_t [u(t^*) ]_{2+\g} <0$. This would establish the bound $[u(t) ]_{2+\g} <R$ on the entire interval of existence, and hence extension to a global solution.

\medskip

\noindent \textsc{Preliminary estimates on $[0,t^*]$}. First we observe two simple bounds:
\begin{equation}\label{e:u12}
[u(t)]_1, [u(t)]_2 <  R^{-\frac{6}{\a}} e^{-\frac{c_0 t}{R}}, \text{ for all } t\leq t^*,
\end{equation} 
provided $R$ and $N$ are sufficiently large. Indeed, in view of \eqref{e:A} and $\cM \geq 1/R$,
\[
 [u]_1 \leq A^{\frac{1+\g}{2+\g}} [u]_{2+\g}^{\frac{1}{2+\g}} \leq R^{1-N/2} e^{-c_0 t / R} <R^{-\frac{6}{\a}} e^{-\frac{c_0 t}{R}},
 \]
and similarly,
\[
[u]_2 \leq A^{\frac{\g}{2+\g}} [u]_{2+\g}^{\frac{2}{2+\g}} < R^{1 - N  \frac{\g}{2+\g}} e^{-c_0 t/R} \leq R^{-\frac{6}{\a}} e^{-c_0 t/R}.
\]
Next, we provide bounds on the density. Let us denote $\lowr$ and $\uppr$ the minimum and maximum of $\rho$, respectively.  Denote $d = \n \cdot u$. The classical estimates imply
\[
\lowr_0 \exp\left\{-  \int_0^t |d(s)|_\infty \ds \right\} \leq \lowr(t),\quad  \uppr(t) \leq \uppr_0 \exp\left\{ \int_0^t |d(s)|_\infty \ds \right\} .
\]
By \eqref{e:u12}, $|d|_\infty \leq R^{-3} e^{-c_0 s / R}$. Thus,
\[
 \int_0^t |d(s)|_\infty \ds \leq c R^{-2} \leq \ln 2,
 \]
Hence, we obtain the estimates
\begin{equation}\label{e:rrRR}
\frac{1}{2R} \leq \lowr(t),\quad  \uppr(t) \leq 2R.
\end{equation}
To get similar bounds for higher order derivatives of $\rho$ we resort to the $e$-quantity. Note that the right hand side of the e-equation is bounded by 
\[
| (\n \cdot u)^2 -  \tr(\n u)^2 | \leq c [u]_1^2 \lesssim R^{-6 } e^{-c_0 t /R}.
\]
From \eqref{e:e} we thus obtain
\[
\ddt |e|_\infty \leq R^{-3} e^{-c_0 t / R} |e|_\infty + R^{-6} e^{-c_0 t /R}.
\]
Again, by Gr\"onwall, and using that $|e_0|_\infty <R$,
\begin{equation}\label{e:eR}
	|e|_\infty \leq 2R,
\end{equation}
The estimate for $\n e$ follows similar calculation. Differentiating and performing standard estimates we obtain
\[
\ddt [e]_1 \lesssim [u]_1 [e]_1 + [u]_2 |e|_\infty + c [u]_1 [u]_2 .
\]
Using \eqref{e:u12} we obtain
\[
\ddt [e]_1 \lesssim  R^{-3} e^{-c_0 t / R} [e]_1 + 2 R^{-2} e^{-c_0 t/R} + c  R^{-3} e^{-c_0 t/R}.
\]
Note that initially $[e_0]_1 \leq [u_0]_2 + [\rho_0]_3 < 2R$. So, by Gr\"onwall, 
\[
[e]_1 \leq 4R,
\]
and hence, for $\a\neq 1$, we obtain
\begin{equation}\label{e:r1a}
[\rho]_{1+\a} \leq 5R,
\end{equation}
while for $\a=1$,
\[
[\cL_1 \rho]_1 \leq 5R.
\]
The latter does not guarantee a bound in $W^{2,\infty}$, however it implies bounds in other border-line classes such as Zygmund or Besov $B^2_{\infty,\infty}$. It will be sufficient for what follows to reduce the exponent $2$ by $\g>0$, which will ultimately depend on $\a$ only, and quote the case $\a\geq 1$ as 
\begin{equation}\label{e:rR}
[\rho]_{1+\a - \gamma} \leq C_\a R.
\end{equation}

\medskip

\noindent \textsc{Nonlinear bound on dissipation}.  We establish another auxiliary bound on the dissipation term similar to the nonlinear maximum principle estimate of Constantin and Vicol \cite{CV2012}, see also \cite{CCotiV2016}.
Denote  
\[
D_\a f (x) = \int_{\R^n}  |f(x+z) - f(x)|^2  \frac{\dz}{|z|^{n+\a}}.
\]
\begin{lemma} There is an absolute constant $c_0>0$ such that
    \begin{equation}\label{e:D}
    D_\a   \d_h^3  u (x) \geq  c_0 \frac{|  \d_h^3  u (x) |^{2 + \a}}{ [u]_2^\a |h|^{3 \a}}.
    \end{equation}
\end{lemma}
\begin{proof}
    Let us fix a smooth cut-off function $\psi$, and fix an $r>0$. We obtain
    \[
    \begin{split}
    D_\a   \d_h^3  u (x) & \geq  \int |\d_z \d_h^3  u (x)|^2 \frac{1-\psi(z/r)}{|z|^{n+\a}} \dz \\ 
    & \geq  \int (|\d_h^3  u (x)|^2 - 2 \d_h^3  u (x)  \d_h^3  u (x+z)) \frac{1-\psi(z/r)}{|z|^{n+\a}} \dz\\
    &\geq | \d_h^3  u (x)|^2 \frac{1}{r^\a}  - 2  \d_h^3  u (x) \int \d_h^3  u (x+z) \frac{1-\psi(z/r)}{|z|^{n+\a}} \dz
    \end{split}
    \]
    Notice that  
    \[
    \d_h^3  u (x+z) =  \int_0^1\int_0^1\int_0^1 \n^3_z u(x+z+(\th_1+\th_2+\th_3)h)  (h,h,h)  \dth_1\dth_2 \dth_3.
    \]
Integrating by parts in $z$ once, and using the bound 
\[
\left| \n_z \frac{1-\psi(z/r)}{|z|^{n+\a}} \right| \leq \frac{c}{|z|^{n+\a+1}} \chi_{|z|>r},
\]
 we obtain
    \[
    \left| \int \d_h^3  u (x+z) \frac{1-\psi(z/r)}{|z|^{n+\a}} \dz \right| \leq C [u]_2 \frac{|h|^3}{r^{\a+1}}.
    \]
    We continue with the estimate:
    \[
    D_\a   \d_h^3  u (x) \geq | \d_h^3  u (x)|^2 \frac{1}{r^\a} - C [u]_2 |\d_h^3  u (x)| \frac{|h|^3}{r^{\a+1}}.
    \]
    Optimizing in $r$ yields the result.
\end{proof}

In view of the preliminary estimates we established and the assumption of the breakthrough scenario, we consequently obtain
\begin{equation}\label{e:diss}
\frac{1}{|h|^{4+ 2\g}}	   D_\a   \d_h^3  u (x) \geq \frac{R^{8+\a}}{|h|^{\a(1-\g)}}. 
\end{equation}

\medskip

\noindent \textsc{Main estimates.} With all ingredients at hand we are now ready to use the equation to make estimate on the derivative of $[u]_{2+\g}$.  Let  $(x,h)\in \T^n$ be a pair for which the supremum \eqref{e:holder} is attained.  We write the equation for the third order difference:
\begin{equation}
\p_t \d^3_h u +  \d^3_h ( u   \n u)  = \int_\R \d_h^3[\rho(\cdot +z)( u(\cdot +z) - u(\cdot))] \frac{\dz}{|z|^{n+\a}}. 
\end{equation}
Denote
\begin{equation}
\begin{split}
    B& = \d^3_h ( u   \n u), \\
    I &= \int_\R \d_h^3[\rho(\cdot +z)( u(\cdot +z) - u(\cdot))] \frac{\dz}{|z|^{n+\a}}.
\end{split}
\end{equation}
We will be testing the equation with $\d_h^3 u(x)/|h|^{4+2\g}$. To expand the bilinear terms we make use of the product formula
\[
\d^3_h(fg) = \d_h^3 f \tau_{3h} g + 3 \d_h^2 f \d_h \tau_{2h} g + 3 \d_h f \d_h^2 \tau_h g + f \d_h^3 g.
\]
For the B-term we obtain
\[
B = \d_h^3 u \tau_{3h} \n u + 3 \d_h^2 u \d_h \tau_{2h} \n u + 3 \d_h u \d_h^2 \tau_h \n u + u \n \d_h^3 u.
\]
Note that the last term vanishes due to criticality. Thus, we can estimate 
\[
\frac{1}{|h|^{2+\g}} |B| \leq [u]_{2+\g} [u]_1  + 3|h|^{1-\g}[u]_2^2 + 3 [u]_1 [u]_{2+\g} \lesssim [u]_{2+\g} [u]_1 +  |h|^{1-\g} [u]_2^2.
\]
Multiplying by another $[u]_{2+\g} = R$ and using \eqref{e:u12} we obtain
\begin{equation}\label{e:B}
\frac{[u]_{2+\g}}{|h|^{2+\g}} |B| \lesssim R^{-1} + R^{-5} < 1.
\end{equation}
We now turn to the I-term which contains dissipation. 
The integrand is given by $\d^3_h [ \tau_z \rho\, \d_z u]$. So, we expand similarly using commutativity $\d_h \d_z = \d_z \d_h$:
\begin{equation}
\begin{split}
\d^3_h [ \tau_z \rho\, \d_z u] = \d_h^3 \tau_z \rho\, \tau_{3h} \d_z u + 3 \d_h^2 \tau_z \rho\, \tau_{2h} \d_h \d_z u+ 3 \d_h \tau_z \rho\, \tau_h \d_h^2 \d_z u + \tau_z \rho\,  \d_z \d_h^3  u.
\end{split}
\end{equation}
Multiplying upon $\d_h^3 u$ the last term becomes dissipative:
\[
\tau_z \rho\,  \d_z \d_h^3  u  \, \d_h^3  u  \leq - \frac12  \underline{\rho}\, |\d_z \d_h^3  u |^2.
\]
Dividing by $|h|^{4+2 \g}$ and using \eqref{e:diss} we obtain
\begin{equation}\label{e:Duh}
\frac{1}{2|h|^{4+2 \g}}  \underline{\rho} D_\a   \d_h^3  u (x) \geq \frac{ R^8}{ |h|^{\a(1-\g)}}.
\end{equation}
 At this point it is clear that the transport term estimated in \eqref{e:B} is completely absorbed into dissipation at the critical time $t^*$:
\[
\p_t [u]_{2+\g}^2 \leq   -  \frac{R^7}{ |h|^{\a(1-\g)}} + \frac{\d_h^3 u(x)}{|h|^{4+2\g}} II.
\]
Here $II$ contains all the remaining three terms of $I$:
\[
II =  \int_{\R^n}[ \d_h^3 \tau_z \rho\, \tau_{3h} \d_z u + 3 \d_h^2 \tau_z \rho\, \tau_{2h} \d_h \d_z u+ 3 \d_h \tau_z \rho\, \tau_h \d_h^2 \d_z u ]  \frac{\dz}{|z|^{n+\a}} = II_1+3 II_2+3 II_3.
\]
Let us now turn to estimates on each of the remaining $II_i$ terms.  Specifically, we will be aiming to obtain bounds of the form 
\begin{equation}\label{e:aim}
	\frac{1}{|h|^{2+\g}} |II_i| \lesssim \frac{ |h|^\e }{ |h|^{\a(1-\g)}}
\end{equation}
for some $\e>0$ provided $\g$ is sufficiently small. This consequently makes the dissipation term absorb all the remaining terms in the equation.

We start with $II_2$. For $\a<1$, we use \eqref{e:u12} and \eqref{e:r1a} to obtain
\[
\begin{split}
|\d_h^2 \tau_z \rho | & \leq [\rho]_{1+\a} |h|^{1+\a}  \lesssim R|h|^{1+\a}\\
|\tau_{2h} \d_h \d_z u |& \leq [u]_2 |h| \min\{|z|,1\} \lesssim R^{-1}  |h| \min\{|z|,1\}.
\end{split}
\]
Thus, the singularity is removed and we obtain 
\[
\frac{1}{|h|^{2+\g}}|II_2| \lesssim |h|^{\a - \g},
\]
which clearly implies \eqref{e:aim} for sufficiently small $\g$.  In the case $\a \geq 1$ we first symmetrize
\[
II_2 = \frac12 \int_{\R^n}[ \d_h^2 (\tau_z - \tau_{-z})\rho\, \tau_{2h} \d_h \d_z u+\d_h^2 \tau_z \rho\, \tau_{2h} \d_h (\d_z + \d_{-z}) u ]  \frac{\dz}{|z|^{n+\a}}.
\]
For the first part we use \eqref{e:rR}:
\[
\begin{split}
|\d_h^2 (\tau_z - \tau_{-z})\rho | & \leq R \min\{|h|^{1+\a-\g}, |h|^{\a-\g} |z|\}  \\
|\tau_{2h} \d_h \d_z u |& \leq R^{-1}  |h| \min\{|z|,1\}.
\end{split}
\]
Hence,
\[
\frac{1}{|h|^{2+\g}} \int_{\R^n} | \d_h^2 (\tau_z - \tau_{-z})\rho\, \tau_{2h} \d_h \d_z u |   \frac{\dz}{|z|^{n+\a}} \leq  \frac{|h|^{1+\a - \g}}{|h|^{2+\g}}  \leq  \frac{ |h|^{\a -2 \g -1 + \a(1-\g)} }{ |h|^{\a(1-\g)}}.
\]
Clearly, $\a - 2 \g -1 + \a(1-\g)>0$.  Lastly, using that $(\d_z + \d_{-z}) u $ is the second difference,
\[
| \d_h^2 \tau_z \rho\, \tau_{2h} \d_h (\d_z + \d_{-z}) u | \leq |h|^{2-\g} \min\{ |z|^2,1\}
\]
we obtain
\[
\frac{1}{|h|^{2+\g}} \int_{\R^n} | \d_h^2 \tau_z \rho\, \tau_{2h} \d_h (\d_z + \d_{-z}) u |  \frac{\dz}{|z|^{n+\a}} \leq \frac{|h|^{2-\g}}{|h|^{2+\g}} \leq  \frac{h^{\a(1-\g) - 2\g }}{|h|^{\a(1-\g)}}.
\]
This completes the bounds on $II_2$. 

As to $II_3$ we proceed similarly.  For $\a<1$, we use
\[
|\d_h \tau_z \rho\, \tau_h \d_h^2 \d_z u | \leq |h|^2 \min\{|z|,1\}.
\]
Hence,
\[
\frac{1}{|h|^{2+\g}} |II_3| \lesssim \frac{|h|^2}{|h|^{2+\g}}  \leq  \frac{h^{\a(1-\g) - \g }}{|h|^{\a(1-\g)}}.
\]
The power in the numerator is positive for sufficiently small $\g$.  For $\a\geq 1$,  we again symmetrize first
\[
II_3 = \frac12 \int_{\R^n}[ \d_h (\tau_z - \tau_{-z})\rho\, \tau_{h} \d^2_h \d_z u+\d_h \tau_z \rho\, \tau_{h} \d^2_h (\d_z + \d_{-z}) u ]  \frac{\dz}{|z|^{n+\a}}.
\]
Thus,
\[
\begin{split}
|\d_h (\tau_z - \tau_{-z})\rho\, \tau_{h} \d^2_h \d_z u |  & \leq \min\{|h|^{3+\a - \g}, |h|^{1+ \a} |z|^2 \} \\
| \d_h \tau_z \rho\, \tau_{h} \d^2_h (\d_z + \d_{-z}) u | & \leq \min\{|h|^3, |h| |z|^2\}
\end{split}
\]
The first term results in an estimate as before. For the second we split the integration into regions $|z|<r$ and $|z|>r$ to obtain the bound by $|h| r^{2-\a} + |h|^3 r^{-\a}$. Setting $r = |h|$ leads to a further bound by $|h|^{3-\a}$ which implies the desired \eqref{e:aim}.

Estimates on $II_1$ are somewhat more involved. For the case $\a \geq 1$ we symmetrize:
\[
 II_1 = \frac12 \int_{\R^n}[ \d_h^3 ( \tau_z \rho - \tau_{-z}\rho) \, \tau_{3h} \d_z u +  \d_h^3 \tau_z \rho\, \tau_{3h} (\d_z u + \d_{-z} u) ]  \frac{\dz}{|z|^{n+\a}}.
\]
For the first half we use 
\[
| \d_h^3 ( \tau_z \rho - \tau_{-z}\rho) \, \tau_{3h} \d_z u | \leq |h|^{\a-\g} \min\{|z|^2, |h|\}.
\]
For the second half we use
\[
| \d_h^3 \tau_z \rho\, \tau_{3h} (\d_z u + \d_{-z} u) | \leq |h|^{1+\a-\g} \min\{|z|^2,1\},
\]
so, this is esimated as before. One can see that in fact the estimates above extend to the range $\a>\frac12$, but not all the way to zero. The problem is that the density takes all the variations in $h$ and not fully uses them, while $u$ cannot directly contribute. So, we will swap one $h$-difference back onto $u$.   We start from the original formula
\[
II_1 =   \int_{\R^n}\d_h^3 \tau_z \rho(x) \, \tau_{3h} \d_z u (x) \frac{\dz}{|z|^{n+\a}}.
\]
Over the region $|z|<10 |h|$ we estimate directly using the same cut-off function $\psi$ as earlier:
\[
 \int_{\R^n} | \d_h^3 \tau_z \rho(x) \, \tau_{3h} \d_z u (x) | \, \psi\left(\frac{z}{10|h|}\right) \frac{ \dz}{|z|^{n+\a}} \leq \int_{|z|<10 |h| } |h|^{1+\a} \frac{\dz}{|z|^{n+\a-1}} \lesssim |h|^2,
 \]
this results in \eqref{e:aim}.  For the remaining part, let us denote for clarity $f = \d_h^2  \rho$. So, $\d_h^3 \tau_z \rho(x) = f(x+h+z) - f(x+z)$.  We write
 \[
 \begin{split}
 & \int_{\R^n}  (f(x+h+z) - f(x+z)) \, \tau_{3h} \d_z u (x) |\frac{(1-\psi(\frac{z}{10|h|})) \dz}{|z|^{n+\a}}  \\
 &=  \int_{\R^n} f(x+z)  \left( \tau_{3h} \d_{z-h} u (x)  \frac{(1-\psi(\frac{z-h}{10|h|}))}{|z - h|^{n+\a}}  - \tau_{3h} \d_{z} u (x)  \frac{(1-\psi(\frac{z}{10|h|}))}{|z |^{n+\a}} \right) \dz \\
 & = \int_{\R^n} f(x+z)  \tau_{3h} (\d_{z-h} u (x) - \d_{z} u (x)) \frac{(1-\psi(\frac{z-h}{10|h|}))}{|z - h|^{n+\a}}  \dz     \\
 & -  \int_{\R^n} f(x+z) \tau_{3h} \d_{z} u (x) \left( \frac{(1-\psi(\frac{z-h}{10|h|}))}{|z - h|^{n+\a}}   -   \frac{(1-\psi(\frac{z}{10|h|}))}{|z |^{n+\a}}  \right) \dz 
 \end{split}
 \]
Note that the integrals are still supported on $|z| > 9|h|$, where $|z - h| \sim |z|$.  Estimating the first part we use 
\[
\begin{split}
|\d_{z-h} u (x) - \d_{z} u (x)| &=| u(x+z - h) - u(x+z)| \leq |h| \\
|f(x+z)| & \leq |h|^{1+\a}.
\end{split}
\]
thus,
\[
\left| \int_{\R^n} f(x+z)  \tau_{3h} (\d_{z-h} u (x) - \d_{z} u (x)) \frac{(1-\psi(\frac{z-h}{10|h|}))}{|z - h|^{n+\a}}  \dz   \right| \leq |h|^{2+\a} \int_{|z|\geq |h|} \frac{\dz}{|z |^{n+\a}} \leq |h|^2,
\]
which implies \eqref{e:aim}.  Finally, for the second part we use
\[
\left| \frac{(1-\psi(\frac{z-h}{10|h|}))}{|z - h|^{n+\a}}   -   \frac{(1-\psi(\frac{z}{10|h|}))}{|z |^{n+\a}}  \right| \leq |h| \frac{I_{|z|>9|h|}}{|z - \th h |^{n+\a+1}}  \lesssim   \frac{  |h| I_{|z|>9|h|}}{|z  |^{n+\a+1}} ,
\]
and 
\[
| f(x+z) \tau_{3h} \d_{z} u (x) | \leq |h|^{1+\a} |z|.
\]
Integration produces the same estimate as for the first part.

 We have established that $\p_t [u(t^*) ]^2_{2+\g} <0$ at the critical time, which finishes the proof.
 
 \medskip

\noindent \textsc{Flocking}.  We have constructed solutions which enjoy the global bounds \eqref{e:u12} and \eqref{e:rR}, which in turn implies $|\n \rho|_\infty < C R$. 
Arguing as in \cite{ST2},  we denote $\widetilde{\rho}(x,t) := \rho(x+ t \bar{u},t)$:
 \[
\p_t \widetilde{\rho} = - (u - \bar{u}) \cdot \n \widetilde{\rho} - d \widetilde{\rho} ,
\]
where all the $u$'s are evaluated at $x+ t \bar{u}$. According to the established bounds, the right hand side is exponentially decaying quantity in  $L^\infty$:
\[ 
|(u - \bar{u}) \cdot \n \widetilde{\rho} + d \widetilde{\rho}|_\infty
\leq C e^{-\d t}.
\]
Hence, $\widetilde{\rho}(t) $ is Cauchy as $t \ra \infty$, and hence there exists a unique limiting state, $\rho_\infty(x)$, such that
\[
| \widetilde{\rho} (\cdot,t) - \rho_\infty(\cdot)|_\infty < C_1 e^{-\d t}.
\]
Shifting back to labels $x$, $\bar{\rho}(x,t)=\rho_\infty(x-t\bar{u})$, we have
\[
| \rho(\cdot,t ) - \bar{\rho}(\cdot,t)|_\infty < C_1 e^{-\d t}.
\]
We also have  $\bar{\rho} \in W^{1+\a - \g,\infty}$ by compactness. Using again \eqref{e:rR} and by interpolation we have convergence in the $W^{1,\infty}$-metric as well:
\[
[ \rho(\cdot,t ) - \bar{\rho}(\cdot,t)]_1 < C_2 e^{-\d t}.
\]

\medskip

\noindent \textsc{Stability}. The computation above shows that in fact the limiting flock $r_\infty$ differs little from initial density $r_0$ under the conditions of \thm{t:stab}. Indeed, setting $R$ such that $\e = 1/R^N$ (here $\e>0$ is small), we obtain via \eqref{e:u12},
\[
| \p_t \tilde{r} |_\infty \leq C R^{-2} e^{-c_0 t /R}.
\]
Hence, $| r_\infty - r_0 |_\infty \leq \frac{C}{c_0 R} = \e^\th$. Since $|r_0 - \rho_\infty|<\e$, this finishes the result.

%\bibliographystyle{plain}
%\bibliography{collective,shvydkoy,fractional}

\begin{thebibliography}{10}
    
    \bibitem{CCTT2016}
    Jos\'e~A. Carrillo, Young-Pil Choi, Eitan Tadmor, and Changhui Tan.
    \newblock Critical thresholds in 1{D} {E}uler equations with non-local forces.
    \newblock {\em Math. Models Methods Appl. Sci.}, 26(1):185--206, 2016.
    
    \bibitem{CCotiV2016}
    Peter Constantin, Michele Coti~Zelati, and Vlad Vicol.
    \newblock Uniformly attracting limit sets for the critically dissipative {SQG}
    equation.
    \newblock {\em Nonlinearity}, 29(2):298--318, 2016.
    
    \bibitem{CV2012}
    Peter Constantin and Vlad Vicol.
    \newblock Nonlinear maximum principles for dissipative linear nonlocal
    operators and applications.
    \newblock {\em Geom. Funct. Anal.}, 22(5):1289--1321, 2012.
    
    \bibitem{CS2007a}
    Felipe Cucker and Steve Smale.
    \newblock Emergent behavior in flocks.
    \newblock {\em IEEE Trans. Automat. Control}, 52(5):852--862, 2007.
    
    \bibitem{DKRT2017}
    Tam Do, Alexander Kiselev, Lenya Ryzhik, and Changhui Tan.
    \newblock Global regularity for the fractional euler alignment system.
    \newblock 2017.
    
    \bibitem{Ha2014}
    Seung-Yeal Ha, Moon-Jin Kang, and Bongsuk Kwon.
    \newblock A hydrodynamic model for the interaction of {C}ucker-{S}male
    particles and incompressible fluid.
    \newblock {\em Math. Models Methods Appl. Sci.}, 24(11):2311--2359, 2014.
    
    \bibitem{HeT2017}
    Siming He and Eitan Tadmor.
    \newblock Global regularity of two-dimensional flocking hydrodynamics.
    \newblock {\em C. R. Math. Acad. Sci. Paris}, 355(7):795--805, 2017.
    
    \bibitem{MT2014}
    Sebastien Motsch and Eitan Tadmor.
    \newblock Heterophilious dynamics enhances consensus.
    \newblock {\em SIAM Rev.}, 56(4):577--621, 2014.
    
    \bibitem{ST3}
    Roman Shvydkoy and Eitan Tadmor.
    \newblock Eulerian dynamics with a commutator forcing {III}: {F}ractional
    diffusion of order $0<\a<1$.
    \newblock to appear in Physica D.
    
    \bibitem{ST1}
    Roman Shvydkoy and Eitan Tadmor.
    \newblock Eulerian dynamics with a commutator forcing.
    \newblock {\em Transactions of Mathematics and Its Applications}, 1(1):tnx001,
    2017.
    
    \bibitem{ST2}
    Roman Shvydkoy and Eitan Tadmor.
    \newblock Eulerian dynamics with a commutator forcing {II}: {F}locking.
    \newblock {\em Discrete Contin. Dyn. Syst.}, 37(11):5503--5520, 2017.
    
    \bibitem{Triebel}
    Hans Triebel.
    \newblock {\em Interpolation theory, function spaces, differential operators}.
    \newblock Johann Ambrosius Barth, Heidelberg, second edition, 1995.
    
    \bibitem{PS2012}
    T~Vicsek and A.~Zefeiris.
    \newblock Collective motion.
    \newblock {\em Phys. Reprints}, 517:71 -- 140, 2012.
    
\end{thebibliography}

\end{document}